\documentclass[12pt]{amsart}

\usepackage{amsmath,amsthm,amssymb}

\usepackage{a4wide}

\usepackage{color}

\newtheorem{theorem}{Theorem}[section]
\newtheorem{lemma}[theorem]{Lemma}
\newtheorem{proposition}[theorem]{Proposition}
\newtheorem{corollary}[theorem]{Corollary}

\theoremstyle{remark}
\newtheorem{remark}[theorem]{Remark}

\newtheorem{example}[theorem]{Example}

\newcommand\fieldsetc{\mathbb}

\renewcommand{\H}{\ensuremath{\fieldsetc{H}}}
\newcommand{\C}{\ensuremath{\fieldsetc{C}}}
\newcommand{\R}{\ensuremath{\fieldsetc{R}}}

\newcommand{\Ca}{\ensuremath{\fieldsetc{O}}}

\newcommand{\g}[1]{\ensuremath{\mathfrak{#1}}}

\DeclareMathOperator{\id}{id}
\DeclareMathOperator{\Ad}{Ad}

\DeclareMathOperator{\Exp}{Exp}

\DeclareMathOperator{\spann}{span}

\DeclareMathOperator{\eS}{S}

\DeclareMathOperator{\HH}{H}

\newcommand\liegr{\sf}

\newcommand{\GL}[1]{\mbox{${\liegr GL}(#1)$}}
\newcommand{\SU}[1]{\mbox{${\liegr SU}(#1)$}}
\newcommand{\SUxU}[2]{\mbox{${\liegr S(U}(#1)\times{\liegr U}(#2))$}}
\newcommand{\U}[1]{\mbox{${\liegr U}(#1)$}}
\newcommand{\Sp}[1]{\mbox{${\liegr Sp}(#1)$}}
\newcommand{\SO}[1]{\mbox{${\liegr SO}(#1)$}}

\newcommand{\Spin}[1]{\mbox{${\liegr Spin}(#1)$}}

\newcommand{\G}{\mbox{$\liegr{G}_2$}}
\newcommand{\E}[1]{\mbox{$\liegr{E}_{#1}$}}

\newcommand{\Sph}[1]{\mbox{$\liegr{S}^{#1}$}}

\begin{document}
\title{On homogeneous manifolds whose isotropy actions are polar}

\author[J.\ C.\ D\'{\i}az-Ramos]{Jos\'{e} Carlos D\'{\i}az-Ramos}
\author[M.\ Dom\'{\i}nguez-V\'{a}zquez]{Miguel Dom\'{\i}nguez-V\'{a}zquez}
\author[A.\ Kollross]{Andreas Kollross}

\address{Department of Mathematics, University of Santiago de Compostela, Spain}
\address{Instituto de Ciencias Matem\'{a}ticas (CSIC-UAM-UC3M-UCM) Madrid, Spain}
\address{Institut f\"{u}r Geometrie und Topologie, Universit\"{a}t Stuttgart, Germany}

\email{josecarlos.diaz@usc.es}
\email{miguel.dominguez@icmat.es}
\email{kollross@mathematik.uni-stuttgart.de}

\begin{abstract}
We show that simply connected Riemannian homogeneous spaces of compact semisimple Lie groups with polar isotropy actions are symmetric, generalizing results of Fabio Podest\`{a} and the third named author. Without assuming compactness, we give a classification of Riemannian homogeneous spaces of semisimple Lie groups whose linear isotropy representations are polar. We show for various such spaces that they do not have polar isotropy actions. Moreover, we prove that Heisenberg groups and non-symmetric Damek-Ricci spaces have non-polar isotropy actions.
\end{abstract}




\thanks{The first and second authors have been supported by projects MTM2016-75897-P (AEI/FEDER, UE) and ED431F 2017/03 (Xunta de Galicia, Spain). The second author has received funding from the ICMAT Severo Ochoa project SEV-2015-0554 (MINECO, Spain), and from the European Union's Horizon 2020 research and innovation programme under the Marie Sklodowska-Curie grant agreement No.~745722.}

\maketitle

%
%

\section{Introduction}
\label{sect:Intro}

Symmetric spaces are a central class of examples in Riemannian geometry. It has often been a fruitful approach to study to what extent certain properties of symmetric spaces continue to hold for more general classes of Riemannian manifolds. This point of view has, for example, led to the exploration of g.o.-spaces and the discovery of Damek-Ricci harmonic spaces.

In this article, we study the question whether another well-known property of Riemannian symmetric spaces, namely the polarity of their isotropy actions, holds for more general types of homogeneous manifolds. We say that an isometric Lie group action on a Riemannian manifold is \emph{polar} if there exists an embedded closed submanifold~$\Sigma$ which intersects all orbits of the group action and such that at each point~$p \in \Sigma$ the intersection of the orbit through~$p$ with the submanifold~$\Sigma$ is orthogonal. Such a submanifold~$\Sigma$ is called a \emph{section} of the polar action. A polar action is called \emph{hyperpolar} if it has a section which is flat in its induced Riemannian metric.

The isotropy actions of Riemannian symmetric spaces are well-known to be polar, in fact they are hyperpolar. Indeed, the maximal flats of the symmetric space, i.e.\ the maximal totally geodesic and flat subspaces, provide sections. More precisely, any maximal flat containing~$p$ is a section for the \emph{isotropy action} at~$p$, i.e.\ the action of~$H$ on~$M=G/H$ given by $h \cdot gH = hgH$, where $G$ is (the connected component of the identity of) the isometry group of~$M$  and $H=G_p$ its isotropy subgroup at~$p$. For example, a compact Lie group~$L$ with biinvariant metric is a Riemannian symmetric space, its isotropy action at the identity element~$e$ is given by the action of~$L$ on itself by conjugation, and the maximal tori are the sections for this action.

In order to determine if polar isotropy also occurs in other classes of Riemannian homogeneous spaces, it is a natural first step to consider their \emph{isotropy representation}. The isotropy representation of $H$ on $M=G/H$ is the linear action of~$H$ on~$T_pM$, where $h \cdot v = D_ph(v)$, i.e.\ the action given by the differentials of the isometries $h \in H$, which are linear endomorphisms of~$T_pM$. Since slice representations of polar actions are polar, cf.~\cite[Thm.~4.6]{pt1}, it follows that the linear isotropy representation of~$H$ on the tangent space~$T_pM$ is polar whenever the isotropy action of~$H$ on~$G/H$ is polar. However, the converse is not true, as was shown in~\cite{kp}; see Propositions~\ref{prop:hermnp} and~\ref{prop:quatex} below for more counterexamples.

Our main result is the following.

\begin{theorem}\label{th:main}
Let $G$ be a simply connected compact semisimple Lie group and let $H$ be a closed connected non-trivial proper subgroup.
Let the homogeneous manifold $M = G/H$ be endowed with a $G$-invariant Riemannian metric~$\mu$. Then the isotropy action of~$H$ on~$G/H$ is polar with respect to~$\mu$ if and only if the Riemannian manifold $(G/H,\mu)$ is a Riemannian symmetric space.
\end{theorem}

This is a generalization of~\cite{kp}, where the statement of Theorem~\ref{th:main} was shown in the special case where $G$ is simple. Furthermore, all $G/H$ with $G$ semisimple and compact with polar linear isotropy representation were classified in~\cite{kp}.

Since actions of cohomogeneity one are polar, our main result can be viewed as a generalization of the well-known fact that two-point homogeneous spaces are rank-one symmetric spaces in the special case of homogeneous spaces of semisimple compact Lie groups.
Other similar results were proved in~\cite{hptt}, \cite{ks}, \cite{ks2} and~\cite{samiou}. In \cite{deng}, invariant Finsler metrics on homogeneous spaces with polar isotropy representations were studied.

We would like to point out here that, although the linear isotropy representation in the non-compact case is just obtained from duality of semisimple real Lie algebras, the isotropy action is much harder to study than in the compact case.  Our technique to show that an action is non-polar in the compact case is to exhibit a non-polar slice representation.  Nevertheless, we show that in the non-compact case there are isotropy actions that are infinitesimally polar (that is, all slice representations are polar) but non-polar.

Finally, there is an interesting class of homogeneous spaces whose isometry group is not semisimple and that has been thoroughly studied: Damek-Ricci spaces and their close relatives, 2-step nilpotent generalized Heisenberg groups.  Damek-Ricci spaces gained recognition because they provide counterexamples to the Lichnerowicz conjecture~\cite{DR}: Damek-Ricci spaces are harmonic (the volume density function is a radial function around each point), but generically non-symmetric.  We show in this paper that symmetric Damek-Ricci spaces are precisely the ones that have polar isotropy and that generalized Heisenberg groups have non-polar isotropy.  Therefore, within the class of Heisenberg groups and Damek-Ricci spaces, Theorem~\ref{th:main} remains valid.
\medskip

This paper is organized as follows. In \S\ref{sect:Prelim} we collect the basic results on the classification of homogeneous spaces with compact Lie groups of isometries whose linear isotropy representation is polar.  In \S\ref{sec:expl} we decide which of these homogeneous spaces have polar isotropy and prove Theorem~\ref{th:main}. Then, in \S\ref{sec:lin} we consider non-compact semisimple Lie groups. We prove a duality result for the linear isotropy representation and show that, unlike the compact case, there are spaces with non-polar, infinitesimally polar isotropy action. We also study some homogeneous spaces whose isometry group is not semisimple, namely, generalized Heisenberg groups and Damek-Ricci spaces. We determine in \S\ref{sect:Heis-DR} which of them have polar linear isotropy representation, and show that, among these spaces, the isotropy action is polar if and only if such a space is symmetric. We conclude the paper with some open problems in \S\ref{sect:open}.

%
%

\section{Preliminaries}
\label{sect:Prelim}

In~\cite{kp} the following classification of homogeneous spaces~$G/H$ of simple compact Lie groups~$G$ with polar isotropy actions was obtained.

\begin{theorem}\label{th:kpmain}
Let $G/H$ be a homogeneous space with $G$ a simple compact connected Lie group and $H$ a closed connected non-trivial subgroup. If the $H$-action on $G/H$ is polar with respect to a $G$-invariant Riemannian metric~$\mu$, then the Riemannian manifold $(G/H,\mu)$ is locally symmetric and the pair $(G,H)$ is, up to local isomorphism, either a symmetric pair or one of the following: $(\SU{n+1},\SU{n})$, $(\Sp{n+1},\Sp{n})$, $(\Sp{n+1},\Sp{n}\times\U1)$, $(\Spin9,\Spin7)$, $(\Spin7,\G)$, $(\G,\SU3)$.
\end{theorem}

Let $G$ be a Lie group with Lie algebra~$\g g$ and $H \subseteq G$ a compact subgroup. We denote by $\Ad_G|_H$ the restriction of the adjoint representation of~$G$ to the subgroup~$H$. By \[\chi(G,H)= \Ad_G|_H \ominus \Ad_H\] we denote the isotropy representation of the homogeneous space~$G/H$, given by the linear action of~$H$ on the tangent space~$T_{eH}G/H$.

\begin{lemma}\label{lm:polinvsc}
Let $\rho \colon G \to \GL{V}$ be a representation of the compact Lie group~$G$ on the finite-dimensional real vector space~$V$. Let $\mu, \nu \colon V \times V \to \R$ be $G$-invariant scalar products. Then the action of~$G$ on~$V$ is polar with respect to~$\mu$ if and only if it is polar with respect to~$\nu$.
\end{lemma}

\begin{proof}
Since representations of compact Lie groups are completely reducible, there is a decomposition $V = V_0 \oplus V_1 \oplus \dots \oplus V_k$ such that $G$ acts trivially on~$V_0$ and irreducibly on~$V_1, \dots, V_k$. Let $\Sigma \subseteq V$ be a section of the $G$-action on the Euclidean vector space~$(V,\mu)$. Then it follows from~\cite[Thm.~4]{dadok} that $\Sigma$ is of the form $\Sigma = V_0 \oplus \Sigma_1 \oplus \dots \oplus \Sigma_k$, where $\Sigma_i \subset V_i$ are linear subspaces. By~\cite[Lemma~2.9]{kollross} we have that the invariant subspaces $V_1, \dots , V_k$ are mutually inequivalent
and it is a consequence of Schur's Lemma that any $G$-invariant scalar product is such that the~$V_0, \dots, V_k$ are mutually orthogonal and $\nu|_{V_i \times V_i} = \lambda_i \, \mu|_{V_i \times V_i}$ with a positive constant~$\lambda_i$ for $i=1,\dots,k$.
It follows that a vector in~$V$ is orthogonal to~$\Sigma$ with respect to~$\mu$ if and only if it is orthogonal to~$\Sigma$ with respect to~$\nu$.
\end{proof}

In particular, the polarity of the linear representation~$\chi(G,H)$ does not depend on the choice of the $G$-invariant Riemannian metric on~$G/H$. (However, the polarity of the $H$-action on~$G/H$ may depend on the choice of the $G$-invariant Riemannian metric on~$G/H$, cf.~\cite[Thms.~14,\,15]{kp}.)

For homogeneous spaces $G/H$ with compact simple~$G$ the following classification of spaces with polar linear isotropy was given in~\cite{kp}.

\begin{theorem}\label{th:kpth2}
Let $G$ be a simply connected simple compact Lie group and $H$ be a closed connected subgroup. Assume the linear isotropy representation $\chi(G,H)$ is polar. Then $(G,H)$ is either one of the pairs given in Theorem~\ref{th:kpmain} or one of the following:
   \begin{enumerate}
     \item  $(\SU{p+q},\SU p\times\SU q)$, where~$2 \le p<q$;
     \item  $(\Spin{2n},\SU n)$, where $n \ge 5$ is odd;
     \item   $(\E6,\Spin{10})$.
   \end{enumerate}
\end{theorem}

Let $G$ be a simple Lie group and let $H \subseteq G$ be a compact connected subgroup such that $(G,K)$ is a Hermitian symmetric pair, where $K = H \cdot \U1$. Let $\rho = \chi(G,K)$ be the isotropy representation of~$G/K$ and let $\rho|_H$ be its restriction to the subgroup~$H$. If the connected components of the orbits of $\rho$ and~$\rho|_H$ are the same, then we say that the $\U1$-factor is \emph{inessential} for the symmetric pair~$(G,K)$, otherwise we say it is \emph{essential}. It follows from \cite{eh} or~\cite[Thm.~6]{kp} that $\rho|_H$ is polar if and only if the $\U1$-factor is inessential.

\begin{remark}\label{rm:essential}
The pairs $(G,H)$ given in Theorem~\ref{th:kpth2}, but not in Theorem~\ref{th:kpmain}, are exactly those pairs where $G$ is a simply connected simple compact Lie group and $H \subseteq G$ a compact connected subgroup such that $(G,H \cdot \U1)$ is a Hermitian symmetric pair of rank~$\ge 2$ where the $\U1$-factor is \emph{inessential}.
(We remark parenthetically that these are also exactly the pairs $(G,H)$ with $G$ simply connected, simple and compact where $(G,H \cdot \U1)$ is a symmetric pair such that the non-compact dual of~$G/(H \cdot \U1)$ is an irreducible Hermitian symmetric space of \emph{non-tube type} and rank~$\ge 2$.)
\end{remark}

A pair of groups $(G,H)$ with $H \subseteq G$ is said to be \emph{decomposable} if $G \cong G_1 \times G_2$ as a direct product of non-trivial factors and $H = H_1 \times H_2$ with $H_i \subseteq G_i$ for $i=1,2$. Otherwise, we say the pair is \emph{indecomposable}.
For $G$ semisimple and compact, the homogeneous spaces~$G/H$ with polar linear isotropy representation were classified in~\cite{kp}. Obviously, it suffices to classify the spaces where the pair~$(G,H)$ is indecomposable. The result can be stated as follows, where we denote the canonical projections by~$\pi_i \colon G_1 \times \dots \times G_m \to G_i$.

\begin{proposition}\label{prop:kpss}
Let $G = G_1 \times \dots \times G_m$ be a simply connected compact Lie group, where the $G_i$ are simple normal factors, and let $H$ be a compact connected subgroup of~$G$ such that the pair~$(G,H)$ is indecomposable. Assume $m \ge 2$.
Then the linear isotropy representation $\chi(G,H)$ of the homogeneous space~$G/H$ is polar if and only if the pair~$(G,H)$ is one of the following (up to automorphisms of~$G$):

\begin{enumerate}
\item $(L \times L, \Delta L)$, where $L$ is a
simple compact Lie group and $\Delta L = \{ (g,g) \mid g \in L \}$;

\item
$(\Sp{a+1} \times \Sp{b+1}, \Sp a \times \Sp1 \times \Sp b)$, where $a \ge 1, b \ge 0$;

\item
$(G,H)$, where each $(G_i, \pi_i(H))$ is either an irreducible Hermitian symmetric pair or $(\Sp  {n+1}, \Sp{n}\cdot\U1)$, the intersection $G_i \cap H$ agrees with the semisimple part of~$\pi_i(H)$ for each~$i$, and the rank of the center of~$H$ is greater or equal than the number of those indices~$i$ where the $\U1$-factor is essential for $(G_i, \pi_i(H))$.
\end{enumerate}
\end{proposition}

We will now state two results which show that homogeneous spaces with polar isotropy actions are locally symmetric under certain additional assumptions.  Note that neither result requires the transitive action of a semisimple Lie group.

\begin{lemma}\label{lm:CohTwo}
Let $M$ be a homogeneous Riemannian manifold on which a Lie group~$H$ of isometries acts polarly with a fixed point~$p$ and with cohomogeneity two. Then $M$ is locally symmetric.
\end{lemma}

\begin{proof}
We show that there exists a two-dimensional totally geodesic subspace~$\Sigma$ with the property that any geodesic of $M$ can be mapped into~$\Sigma$ by an isometry.

Let $\Sigma$ be a section of the polar $H$-action. Then $p \in \Sigma$, since a section meets all orbits and $p$ is a fixed point. Let $\gamma \colon \R \to M$ be a geodesic. Since $M$ is homogeneous, we may assume that $\gamma(0) = p$. Furthermore, we know that the induced action of $H$ on $T_{p}M$ is polar with $T_{p}\Sigma$ as a section, cf.~\cite[Thm.~4.6]{pt1}. Thus there is an isometry $h \in H$ such that $h \cdot \gamma'(0) \in T_{p}\Sigma$. Since $\Sigma$, as a section of a polar action, is totally geodesic \cite[Thm.~3.2]{pt}, this implies that $h \cdot \gamma(\R) \subseteq \Sigma$.

Now it follows from \cite[Thm.~2.3]{ks} that $M$ is locally symmetric.
\end{proof}

In~\cite{discala}, the following result is shown, which is a generalization of~\cite[Thm.~3]{kp}.

\begin{theorem}\label{th:antonio}
Let $M$ be a homogeneous Riemannian manifold. Let $H$ be a Lie group of isometries acting polarly on~$M$. If the $H$-action on $M$ has a fixed point and a section for the $H$-action is a compact locally symmetric space, then $M$ is locally symmetric.
\end{theorem}

%
%

\section{Non-symmetric spaces with polar linear isotropy}
\label{sec:expl}

We will show in this section that the isotropy actions of the spaces in Proposition~\ref{prop:kpss}~(ii) and~(iii) are non-polar by exhibiting non-polar slice representations. Note that for a homogeneous space~$G/H$, the slice representation of the $H$-action on~$G/H$ at the point~$gH$ is equivalent to the action given by $\Ad_G|_{H \cap gHg^{-1}}$ on~$\g g / (\g h + \Ad_G(g)(\g h))$.

\begin{proposition}\label{prop:hermnp}
The isotropy actions of the spaces in Proposition~\ref{prop:kpss}~(iii) are non-polar.
\end{proposition}

\begin{example}\label{ex:SUxSU}
	In order to illustrate the proof of Proposition~\ref{prop:hermnp} we consider the following example:
	\[
	M=\frac{\Sph{2n+1}\times \Sph{2m+1}}{\Sph{1}}\cong
	\frac{\SU{n+1}\times \SU{m+1}}{\SU{n}\times \SU{m}\times \U{1}},
	\]
	where the action of $\Sph{1}=\U{1}$ is diagonal.
	
	Then, $M$ can be written as $M=G/H$, where $G=G_1\times G_2$, $G_1=\SU{n+1}$, $G_2=\SU{m+1}$, and
	\[
	H=\left\{\left.\left(\left(
	\begin{array}{@{\,}c|c@{\,}}
	A	&\\
	\hline
	&	z\\
	\end{array}\right),	
	\left(\begin{array}{@{\,}c|c@{\,}}
	B	&\\
	\hline
	&	z
	\end{array}\right)
	\right)\in G\ \right|
	A\in\U{n},
	B\in\U{m},
	z\in\U{1}
	\right\}.
	\]
	As usual we denote by $\pi_i\colon G\to G_i$ the projection. Then, we have $K_1=\pi_1(H)\cong\liegr{S}(\U{n}\times \U{1})$ and $K_2=\pi_2(H)\cong\liegr{S}(\U{m}\times \U{1})$. If we put $H_i=H\cap G_i$, then $H_1\cong \SU{n}$ and $H_2\cong\SU{m}$. In particular, $K_1\cong H_1\cdot\U{1}$ and $K_2\cong H_2\cdot\U{1}$, and the $\U{1}$-factor is inessential in this case.

	We take $g_1 \in G_1$ such that
	\[
	g_1K_1g_1^{-1} = \left\{ \left.
	\left(\begin{array}{c|c}
	z &    \\ \hline
	& A
	\end{array}\right) \in \SU{n+1}
	\right| A \in \U{n}, z \in \U1 \right\}.
	\]
	The idea is to consider the slice representation of $G_1/K_1$ at $g_1K_1$.  First,
	\[
	g_1 K_1 g_1^{-1} \cap K_1 = \left\{ \left.
	\left(
	\begin{array}{c|c|c}
	z &   &   \\  \hline
	& A  &   \\  \hline
	&   & w
	\end{array}
	\right) \in \SU{n+1}\
	\right|
	\begin{array}{l}
	A \in \U{n-1},  \\[1ex]
	z,w \in \U1
	\end{array} \right\},
	\]
	and the normal space at the point $g_1 K_1$ of $G_1/K_1$ can be identified with
	\[
	V_1:=\nu_{g_1 K_1}(K_1\cdot g_1 K_1)
	\cong\frac{\g{g}_1}{\g{k}_1+\Ad(g_1)\g{k}_1}\cong
	\left\{ \left.
	\left(
	\begin{array}{c|c|c}
	&   & u  \\  \hline
	&     &   \\  \hline
	-\bar u   &
	\end{array}
	\right)  \in \g{su}(n+1)\
	\right|
	u \in \C
	\right\}.
	\]
	The slice representation of $G_1/K_1$ at $g_1K_1$ is then equivalent to the adjoint action of $g_1 K_1 g_1^{-1}\cap K_1$ on $V_1\cong\C$.  This representation is equivalent to $(z,A,w)\cdot u=zu\bar{w}$, with $z$, $w\in\U{1}$, $A\in\U{n-1}$ satisfying $zw\det A=1$, and $u\in\C$.  Hence, the effectivized slice representation is equivalent to the standard representation of~$\U1$ on~$\C$.  A similar argument can be made for $G_2$.
	
	Now let $g=(g_1,g_2)\in G$, and consider the slice representation of $M=G/H$ at $gH$. It turns out that
	\[
	Q:=g H g^{-1}\cap H=
	\left\{\left.\left(\left(
	\begin{array}{@{\,}c|c|c@{\,}}
	z&&\\
	\hline
	& A	&\\
	\hline
	&&	w\\
	\end{array}\right),
	\left(\begin{array}{@{\,}c|c|c@{\,}}
	z&&\\
	\hline
	& B	&\\
	\hline
	&&	w
	\end{array}\right)
	\right)\in G\ \right|
	\begin{array}{@{}l}
	A\in\U{n-1},\\[0ex]
	B\in\U{m-1},\\[0ex]
	z,w\in\U{1}
	\end{array}
	\right\},
	\]
	and the normal space at the point $gH$ of $G/H$ can be identified with $V:=V_1\oplus V_2$, that is,
	\[
	V=\nu_{g H}(H\cdot gH)
	\cong\frac{\g{g}}{\g{h}+\Ad(g)\g{h}}\cong
	\left\{\left.\left(
	\left(\begin{array}{@{\,}c|c|c@{\,}}
	&&u\\
	\hline
	& &\\
	\hline
	-\bar{u} &&\\
	\end{array}\right),
	\left(\begin{array}{@{\,}c|c|c@{\,}}
	&& v\\
	\hline
	& &\\
	\hline
	-\bar{v} &&
	\end{array}\right)
	\right)\in \g{g}\ \right|
	u,v\in\C
	\right\}.
	\]
	Then, the slice representation of $H$ at $gH$ is equivalent to the adjoint action of~$Q$ on $V\cong \C^2$, and this is equivalent to the action $(A,B,z,w)\cdot (u,v)=(zu\bar{w},zv\bar{w})$. Therefore, the slice representation of $G/H$ at $gH$ is orbit equivalent to the diagonal action of $\U{1}$ on $\C^2$, which is non-polar. Indeed, note that $\pi_i(Q)$ acts polarly on each factor $V_i$ via the standard action of $\U{1}$ on $\C$.  Nevertheless, the action of $Q$ on $V$ is diagonal, not orbit equivalent to the product action $\pi_1(Q)\times\pi_2(Q)$ on $V=V_1\oplus V_2$, and hence, by~\cite[Prop.~1~(ii)]{bergmann}, cannot be polar.
\end{example}

Now we turn our attention to the general case.

\begin{proof}[Proof of Proposition~\ref{prop:hermnp}]
We will show that there is an element $g = (g_1, \dots, g_m) \in G$ such that the isotropy action of~$H$ on~$G/H$ has a non-polar slice representation at~$gH$.  This is equivalent to the slice representation of the $gHg^{-1}$-action at~$eH \in G/H$ being non-polar.

We will determine the components $g_i \in G_i$ of the desired element~$g \in G$ for each factor individually. Let $H_i = H \cap G_i$ and let $K_i=H_i \cdot \U1$.  Then, for each~$i$, either $(G_i,K_i)$ is an irreducible Hermitian symmetric pair or $(G_i,K_i) = (\Sp{n+1},\Sp{n}\cdot\U1)$. We deal with several different cases according to the conditions described in Remark~\ref{rm:essential}.

\begin{enumerate}
  \item First assume $(G_i,K_i)$ is a Hermitian symmetric pair such that the $\U1$-factor is essential for the representation $\chi (G_{i}, K_i)$. For these indices~$i$ we choose $g_i=e$, the identity element of~$G_i$.

  \item Now assume $(G_i,H_i)$ is one of the pairs in Theorem~\ref{th:kpth2}, i.e.\ $(G_i,K_i)$ is a Hermitian symmetric pair of rank~$\ge 2$ such that the $\U1$-factor is inessential for the representation $\chi (G_{i}, \pi_{i}(K))$, i.e.\ one of
      \begin{enumerate}
        \item $(\SU{p+q},\SU p\times\SU q), 2 \le p<q$;
        \item $(\Spin{2n},\SU n), n \ge5\hbox{~odd}$;
        \item $(\E6,\Spin{10})$.
      \end{enumerate}
      Choose the element~$g_i$ as in~\cite[Thm.~13]{kp}. Then the slice representation of the $gK_ig^{-1}$-action on~$G_i/K_i$ at~$eK_i$ is given by the following representations:
      \begin{enumerate}
        \item the linear action of~$\SUxU{p}{p}$ on~$\C^{p}\otimes\C^{p}$, i.e.\ the isotropy representation of $\SU{2p}/\SUxU{p}{p}$;
        \item the linear action of~$\U{n-1}$ on $\Lambda^2 \C^{n-1}$, i.e.\ the isotropy representation of $\SO{2m}/\U{m}$, $m =n-1$~even;
        \item the linear action of~$\Spin8 \cdot \U1$ on~$\R^8 \otimes \R^2$, i.e.\ a linear action which is equivalent to the isotropy representation of $\SO{10}/\SO{8}\times\SO{2}$ after effectivization.
      \end{enumerate}
      Note that the above representations are equivalent to isotropy representations of Hermitian symmetric pairs where the $\U1$-factor is essential, see the remarks after Theorem~\ref{th:kpth2}.
      The slice representation of the  $g_iH_ig_i^{-1}$-action on~$G_i/H_i$ at~$eH_i$ then contains as a submodule:
      \begin{enumerate}
        \item the linear action of~$\SU{p}\times\SU{p}$ on~$\C^{p}\otimes\C^{p}$,
        \item the linear action of~$\SU{n-1}$ on $\Lambda^2 \C^{n-1}$,
        \item the linear action of~$\Spin8$ on~$\R^8 \oplus \R^8$,
      \end{enumerate}
      see~\cite[Thm.~13]{kp}.

  \item If $(G_i,K_i)$ is a Hermitian symmetric pair of rank one, which is only the case if we have \[(G_i,K_i) = (\SU{n+1},\SUxU{n}1),\] we assume that
      \[
      K_i = \left\{ \left.
      \left(\begin{array}{@{\,}c|c@{\,}}
       A &    \\ \hline
        & z
      \end{array}\right) \in \SU{n+1}
      \right| A \in \U{n}, z \in \U1 \right\}.
      \]
      In this case let $g_i \in \SU{n+1}$ be such that
      \[
      g_iK_ig_i^{-1} = \left\{ \left.
      \left(\begin{array}{@{\,}c|c@{\,}}
       z &    \\ \hline
        & A
      \end{array}\right) \in \SU{n+1}
      \right| A \in \U{n}, z \in \U1 \right\}.
      \]
      With a similar computation as in Example~\ref{ex:SUxSU} we find that
       \[
      g_iK_ig_i^{-1} \cap K_i = \left\{ \left.
      \left(
      \begin{array}{@{\,}c|c|c@{\,}}
       z &   &   \\  \hline
        & A  &   \\  \hline
        &   & w
      \end{array}
      \right) \in \SU{n+1}
      \right|
      \begin{array}{c}
      A \in \U{n-1},  \\
      z,w \in \U1
      \end{array} \right\},
      \]
      whose slice representation is given by its action on the matrices
       \[
      \left\{ \left.
      \left(
      \begin{array}{@{\,}c|c|c@{\,}}
         &   & u  \\  \hline
        &     &   \\  \hline
        -\bar u   &
      \end{array}
      \right)  \in \g{su}(n+1)
      \right|
      u \in \C
      \right\}.
      \]
      The effectivized slice representation is thus equivalent to the standard representation of~$\U1$ on~$\C$, and hence its restriction to~$g_iH_ig_i^{-1} \cap H_i$ is trivial.

  \item The case where \[(G_i,K_i) = (\Sp{n+1},\Sp{n}\cdot\U1),\] will be treated similarly.
      In this case, let
      \[
      K_i = \left\{ \left.
      \left(\begin{array}{@{\,}c|c@{\,}}
       A &    \\ \hline
        & z
      \end{array}\right) \in \Sp{n+1}
      \right| A \in \Sp{n}, z \in \U1 \right\},
      \]
      assuming that the elements of~$\Sp{n+1}$ are given by quaternionic $(n+1) \times (n+1)$-matrices,
      and choose $g_i \in \Sp{n+1}$ such that
      \[
      g_iK_ig_i^{-1} = \left\{ \left.
      \left(\begin{array}{@{\,}c|c@{\,}}
       z &    \\ \hline
        & A
      \end{array}\right) \in \Sp{n+1}
      \right| A \in \Sp{n}, z \in \U1 \right\}.
      \]
      We obtain
       \[
      g_iK_ig_i^{-1} \cap K_i = \left\{ \left.
      \left(
      \begin{array}{@{\,}c|c|c@{\,}}
       z &   &   \\  \hline
        & A  &   \\  \hline
        &   & w
      \end{array}
      \right) \in \Sp{n+1}
      \right|
      \begin{array}{c}
      A \in \Sp{n-1},  \\
      z,w \in \U1
      \end{array} \right\},
      \]
      whose slice representation is given by its action on the matrices
       \[
      \left\{ \left.
      \left(
      \begin{array}{@{\,}c|c|c@{\,}}
         &   & q  \\  \hline
        &     &   \\  \hline
        -\bar q   &
      \end{array}
      \right)  \in \g{sp}(n+1)
      \right|
      q \in \H
      \right\}.
      \]
      The effectivized slice representation is equivalent to the componentwise action of~$\U1 \times \U1$ on~$\C \oplus \C$ by the standard representations, and hence its effectivized restriction to~$g_iH_ig_i^{-1} \cap H_i$ is given by the restriction to the diagonal subgroup of~$\U1 \times \U1$.
\end{enumerate}

Let now $V_i$ be the slice representation of the $g_iK_ig_i^{-1}$-action on~$G_i/K_i$ at~$eK_i$, and consider the representation $V = V_1 \oplus \dots \oplus V_m$. Its restriction to~$Q:=gHg^{-1} \cap H$ appears as a submodule of the slice representation of the $gHg^{-1}$-action on~$G/H$ at~$eH$. Note that each $V_i$ is by itself a polar representation of~$Q$. 
Assume that the restriction of~$V$ to~$Q$ is polar. Then, by~\cite[Prop.~1~(iii)]{bergmann}, the action of~$Q$ on $V\ominus V_m=V_1\oplus\dots\oplus V_{m-1}$ is orbit equivalent to the action of $Q_{v_m}$ on $V\ominus V_m$, where $Q_{v_m}$ is the isotropy subgroup  of the $Q$-action on $V_m$ at a regular element $v_m\in V_m$. Since by construction the action of the center of $Q$ on each $V_i$ has the vector $0$ as the only fixed point, and the center of each $\pi_i(H)=K_i$ is one-dimensional, we get that the rank of the center of $Q_{v_m}$ is $\max\{0,r-1\}$, where $r$ stands for the rank of the center of $Q$. By an inductive argument, we deduce that the action of $Q$ on $V_1$ is orbit equivalent to the action of the isotropy group $Q_{(v_2,\dots,v_m)}$, where $v_i\in V_i$ is regular, and the rank of the center of $Q_{(v_2,\dots,v_m)}$ is $\max\{0,r-m+1\}$. By construction of the representation $V$, if the rank of the center of $Q_{(v_2,\dots,v_m)}$ is zero, then its action on $V_1$ (which has the same orbits as the action of $g_1H_1g_1^{-1}\cap H_1$) is not orbit equivalent to the $Q$-action on $V_1$ (which has the same orbits as the action of $g_1K_1g_1^{-1}\cap K_1$). Hence we must have $r=m$. But, since our construction also has the property that the centers of $H$ and~$Q$ have the same rank, we get that the rank of the center of $H$ is $m$. However, since the pair~$(G,H)$ is indecomposable, the rank of the center of~$H$ must be less than~$m$. This proves by contradiction that the $H$-action on~$G/H$ is non-polar.
\end{proof}

\begin{proposition}\label{prop:quatex}
	The isotropy action of the following homogeneous space, where the $\Sp1$-factor is diagonally embedded,
	\begin{equation*}\label{eq:asac}
	G/H = \frac{\Sp{a+1}\times\Sp{b+1}}{\Sp{a}\times\Sp{b}\times\Sp{1}},\quad a,b\ge1,
	\end{equation*}
	is non-polar.
\end{proposition}

\begin{proof}
	We will show that there is an element $g \in G$ such that the isotropy action of~$H$ on~$G/H$ has a non-polar slice representation at~$gH.$  This is equivalent to the slice representation of the $gHg^{-1}$-action at~$eH \in G/H$ being non-polar.
	Let
	\[
	H = \left\{\left.
	\left(\left(
	\begin{array}{@{\,}c|c@{\,}}
	A &    \\ \hline
	& q
	\end{array}
	\right),
	\left(
	\begin{array}{@{\,}c|c@{\,}}
	B &    \\ \hline
	& q
	\end{array}
	\right)\right)\right|
	A \in \Sp{a}, B \in \Sp{b}, q \in \Sp{1}
	\right\}.
	\]
	There is an element $g \in G$ such that
	\[
	gHg^{-1} = \left\{\left.
	\left(\left(
	\begin{array}{@{\,}c|c@{\,}}
	q &    \\ \hline
	& A
	\end{array}
	\right),
	\left(
	\begin{array}{@{\,}c|c@{\,}}
	q &    \\ \hline
	& B
	\end{array}
	\right)\right)\right|
	A \in \Sp{a}, B \in \Sp{b}, q \in \Sp{1}
	\right\}.
	\]
	The slice representation at the point~$eH$ of the $gHg^{-1}$-action on~$G/H$ is given by the representation of the group
	\[
	H \cap gHg^{-1} = \left\{\left.
	\left(\left(
	\begin{array}{@{\,}c|c|c@{\,}}
	p &   &   \\  \hline
	& A  &   \\  \hline
	&   & q
	\end{array}
	\right),
	\left(
	\begin{array}{c|c|c}
	p &   &   \\  \hline
	& B  &   \\  \hline
	&   & q
	\end{array}
	\right)\right)
	\right|
	\begin{array}{c}
	A \in \Sp{a-1}, \\
	B \in \Sp{b-1}, \\
	p,q \in \Sp{1}
	\end{array}
	\right\}
	\]
	on the $8$-dimensional normal space to the $gHg^{-1}$-orbit through~$eH$, which is given by
	\[
	\left\{\left.
	\left(\left(
	\begin{array}{@{\,}c|c|c@{\,}}
	&   & x  \\  \hline
	&   &   \\  \hline
	-\bar x &   &
	\end{array}
	\right),
	\left(
	\begin{array}{@{\,}c|c|c@{\,}}
	&   & y  \\  \hline
	&   &   \\  \hline
	-\bar y  &   &
	\end{array}
	\right)\right) \in \g {sp}(a+1) \times \g {sp}(b+1)
	\right|
	x,y \in \H
	\right\}.
	\]
	It is easy to verify that this representation, after its effectivity kernel is factored out, is equivalent to the action of~$\SO4$ on $\R^4 \oplus \R^4$, i.e.\ the direct sum of two copies of the standard representation. It is well known that this representation is non-polar, see~\cite[Lemma~2.9]{kollross}. Since all slice representations of polar actions are polar~\cite[Thm.~4.6]{pt}, this shows that the isotropy action of~$H$ on~$G/H$ is non-polar.
\end{proof}

\begin{proof}[Proof of Theorem~\ref{th:main}]
Assume first that $G$ is simple. It follows from Theorem~\ref{th:kpmain} that $G/H$ is locally symmetric. Since $G$ is simply connected and $H$ is connected, we have that the Riemannian manifold~$(G/H,\mu)$ is simply connected, complete and locally symmetric. Hence $(G/H,\mu)$ is a Riemannian globally symmetric space.

If $G$ is non-simple, then we may assume that $(G,H)$ is one of the pairs as described in parts (i), (ii), or (iii) of Proposition~\ref{prop:kpss}.
We treat these three cases separately:
\begin{enumerate}
\item The spaces as given by part~(i) of the proposition are symmetric. Indeed, they are homogeneous presentations $L \times L / \Delta L$ of simple compact Lie groups~$L$ with biinvariant Riemannian metric. Since these spaces are isotropy irreducible, their $G$-invariant Riemannian metrics are uniquely determined up to scaling.
\item The spaces as in part~(ii) have been shown to have non-polar isotropy in Proposition~\ref{prop:quatex} if $b \ge 1$. The space
\[
G/H=\frac{\Sp{a+1}\times \Sp1}{\Sp a \times \Sp1}
\]
is diffeomorphic to~$\eS^{4a+3}$. Recall that the subgroup $\Sp{a+1}$ of~$G$ also acts transitively on~$\eS^{4a+3}$. Moreover, the space of all $\Sp{a+1}\times \Sp1$-invariant Riemannian metrics on~$G/H$ is a subset of the space of~$\Sp{a+1}$-invariant Riemannian metrics.  Therefore, this case is already covered by Theorem~\ref{th:kpmain}.

\item Finally, the pairs as described in part~(iii) of Proposition~\ref{prop:kpss} have been shown in Proposition~\ref{prop:hermnp} to have non-polar slice representations.
\end{enumerate}

This completes the proof of our main result.
\end{proof}

We obtain the following observation. Recall that an isometric action is called \emph{infinitesimally polar} if all of its slice representations are polar.

\begin{corollary}\label{cor:inf}
Let $G$ be a simply connected compact semisimple Lie group and let $H$ be a closed connected subgroup.
If the isotropy action of~$H$ on~$G/H$ is {infinitesimally polar} with respect to some $G$-invariant metric, then there exists a symmetric $G$-invariant Riemannian metric on~$G/H$.
\end{corollary}

\begin{proof}
Assume $G$ is a simply connected compact semisimple Lie group and $H \subset G$ a closed connected subgroup such that the isotropy action is infinitesimally polar. In particular, the linear isotropy representation $\chi(G,H)$ is polar, as it is the slice representation of $H$ at $eH$.
All such pairs $(G,H)$ are symmetric pairs or are given by Theorems~\ref{th:kpmain} and \ref{th:kpth2} in case $G$~is simple and by Proposition~\ref{prop:kpss} in case $G$~is non-simple. We show that for each such pair a symmetric metric exists or the isotropy action of~$H$ on~$G/H$ has a non-polar slice representation.

If $(G,H)$ is a symmetric pair, then with any $G$-invariant Riemannian metric, $G/H$~is symmetric and the result is obvious in this case.

Consider now the non-symmetric pairs given in~Theorem~\ref{th:kpmain}. For $(\SU{n+1},\SU{n})$, $(\Sp{n+1},\Sp{n}\times\U1)$, $(\Spin9,\Spin7)$, the cohomogeneity of the $H$-action on~$G/H$ is two, for $(\Spin7,\G)$, $(\G,\SU3)$ it is one. Since orthogonal representations of cohomogeneity~one and two are polar, the $H$-actions on~$G/H$ are infinitesimally polar in these cases. For the action of~$\Sp{n}$ on~$\Sp{n+1}/\Sp{n}$, one can directly check that the slice representations at the two singular points are polar. It is well known that for these six pairs~$(G,H)$ the space~$G/H$ carries a symmetric metric.

For the remaining pairs~$(G,H)$ it is shown in~Propositions~\ref{prop:hermnp},~\ref{prop:quatex} above and the proof of~\cite[Th.~13]{kp} that the $H$-action on~$G/H$ has a non-polar slice representation, so there is nothing to check in this case.
\end{proof}

\begin{remark}
The space $\Sph{7}\times\Sph{7}$ obviously admits a symmetric metric. Nevertheless, take the transitive $\Spin{8}$-action on $\Sph{7}\times\Sph{7}$ given by any two inequivalent non-trivial representations of $\Spin{8}$ on $\R^8$.  It has the isotropy group~$\G$, whose isotropy representation contains two equivalent $\G$-modules and is thus not polar by~\cite[Lemma~2.9]{kollross}. Therefore, the converse of Corollary~\ref{cor:inf} does not hold.
\end{remark}

One can see from Theorem~\ref{th:noncompact-nonpolar} in the next section that the corollary above does not hold anymore if the assumption that $G$ is compact is dropped.

%
%

\section{Some remarks on the non-compact semisimple case}
\label{sec:lin}

We first observe that for semisimple~$G$ the question of polar linear isotropy can be reduced to the compact case via duality of symmetric spaces.

\begin{proposition}\label{prop:dual}
Let $G$ be a simply connected semisimple Lie group and let $H$ be a compact connected subgroup such that the pair~$(G,H)$ is indecomposable.
Assume the isotropy representation~$\chi(G,H)$ of~$G/H$ is polar.
Then there is a Cartan decomposition $\g g = \g k + \g p$ such that $\g h \subseteq \g k$.  Moreover, $\chi(G^*, H)$ is polar and the pair~$(G^*,H)$ is indecomposable, where $\g g^* = \g k + i \g p$ and $G^*$ is the simply connected Lie group with Lie algebra~$\g g^*$.
\end{proposition}

\begin{proof}
Since $H$ is a compact subgroup of~$G$, there is a maximal compact subgroup~$K$ of~$G$ containing~$H$. This maximal compact subgroup gives rise to a Cartan decomposition $\g g = \g k + \g p$. Define an involution $\theta \colon \g g \to \g g$ by requiring $\theta|_{\g k}=\id_{\g k}$ and $\theta|_{\g p}=-\id_{\g p}$.
Then the Lie algebra $\g g^* = \g k + \sqrt{-1}\, \g p$ is compact and semisimple. Let $G^*$ be the simply connected Lie group with Lie algebra~$\g g^*$.
We have $\chi(G^*, H) \cong \chi(G^*, K)|_H \oplus \chi(K, H)$ and $\chi(G^*,K) \cong \chi(G, K)$, and from this it follows that the isotropy representation~$\chi(G^*, H)$ is polar, since it is equivalent to the isotropy representation~$\chi(G, H) \cong \chi(G, K)|_H \oplus \chi(K, H)$.

It remains to show that the pair~$(G^*,H)$ is indecomposable. If $\theta$ maps each simple ideal of~$\g g$ onto itself, then the simple ideals of~$\g g^*$ are in one-to-one correspondence to the simple ideals of~$\g g$ and $(G^*,H)$ is indecomposable. Now assume there are at least two different simple ideals~$\g g_\lambda$ and $\g g_\nu$ such that $\theta(\g g_\lambda)=\g g_\nu$. Since $\theta$ is an automorphism of order two, it follows that we also have $\theta(\g g_\nu)=\g g_\lambda$. Now it follows from the classification of symmetric spaces~\cite{helgason} that $\g g^*$ is the direct sum of two isomorphic simple compact Lie algebras and $\g k$ is the diagonal subalgebra. It follows from Proposition~\ref{prop:kpss} that in this case~$\g h = \g k$ and thus $(G^*,H)$ is indecomposable.
\end{proof}

It is a consequence of the above proposition that the classification of homogeneous spaces of simple non-compact Lie groups with polar linear isotropy representation follows almost immediately from the results of~\cite{kp}:

\begin{corollary}\label{cor:ncslin}
Let $G$ be a simply connected non-compact simple Lie group, and let $H$ be a compact subgroup.
Then the linear isotropy representation $\chi(G,H)$ of the homogeneous space~$G/H$ is polar if and only if the pair~$(G,H)$ is either a symmetric pair or one of the following:
$(\SU{n,1},\SU{n})$; $(\Sp{n,1},\Sp{n})$; $(\Sp{n,1},\Sp{n}\times \U1)$; $(\Spin{8,1},\Spin{7})$; $(\SU{p,q},\SU p\times\SU q)$, where~$2 \le p<q$; $(\liegr{SO}^*(2n),\SU n)$, where $n \ge 5$ is odd; $(\liegr{E}_{6(-14)},\Spin{10})$.
\end{corollary}

\begin{proof}
To prove the corollary, it suffices to consider the pairs in Theorem~\ref{th:kpmain} and~\ref{th:kpth2} and to determine, up to automorphisms of~$G$, all symmetric pairs~$(G,K)$ with~$K$ connected, such that $H \subset K$ after conjugation and $\chi(G^*,H) \cong \chi(G,H)$.

Note that the space~$\G / \SU3$ does not have a non-compact Riemannian counterpart, since the non-compact real form of~$\G$ contains the maximal compact subgroup~$\SO4$ and hence does not contain a compact subgroup locally isomorphic to~$\SU3$; an analogous remark applies to the space $\Spin7/\G$.

In case $(G^*,H)=(\Spin9,\Spin7)$, first assume~$H=\Spin7$ is contained in the subgroup~$K=\Spin7\cdot\SO2$ of~$\Spin9$, then it follows that
\[
\chi(G,K)|_{H} \oplus \chi(K,H) \cong \chi(\Spin9,\Spin7\cdot\SO2)|_{\Spin7} \oplus \R,
\]
is the direct sum of two equivalent $7$-dimensional modules and a one-dimensional trivial module; in particular, $\chi(G,H)$ is non-polar by~\cite[Lemma~2.9]{kollross}; thus $G \cong \Spin{8,1}$ is the only remaining possibility by a dimension count.

In all other cases, there is exactly one possibility for such a group~$K$, as can be seen from a decomposition of $\chi(G^*,H)$ into invariant subspaces.
\end{proof}

For the case of semisimple, non-simple Lie groups~$G$ the classification follows from Proposition~\ref{prop:kpss} by the same method. We obtain:

\begin{corollary}\label{cor:nclinpoliso}
Let $G = G_1 \times \dots \times G_m$ be a simply connected Lie group, where the $G_i$ are simple normal factors, and let $H$ be a compact connected subgroup of~$G$ such that the pair~$(G,H)$ is indecomposable. Assume $m \ge 2$.
Then the linear isotropy representation $\chi(G,H)$ of the homogeneous space~$G/H$ is polar if and only if the pair~$(G,H)$ is either a symmetric pair or one of the following (up to automorphisms of~$G$):

\begin{enumerate}

\item
$(G,H)= (\Sp{a+1} \times \Sp{b+1}, \Sp a \times \Sp1 \times \Sp b)$, $a \ge 1, b \ge 0$;

\item
$(G,H)= (\Sp{a,1} \times \Sp{b+1}, \Sp a \times \Sp1 \times \Sp b)$, $a \ge 1, b \ge 0$;

\item
$(G,H)= (\Sp{a,1} \times \Sp{b,1}, \Sp a \times \Sp1 \times \Sp b)$, $a \ge 1, b \ge 1$;

\item
$(G,H)$, where each $(G_i, \pi_i(H))$ is either an irreducible Hermitian symmetric pair or one of $(\Sp{n+1}, \Sp{n}\cdot\U1)$, $(\Sp{n,1}, \Sp{n}\cdot\U1)$, the intersection $G_i \cap H$ agrees with the semisimple part of~$\pi_i(H)$ for each~$i$, and the rank of the center of~$H$ is greater or equal than the number of those indices~$i$
where the $\U1$-factor is essential for $(G_i, \pi_i(H))$.
\end{enumerate}
\end{corollary}


In the remainder of this section, we show that for the non-compact analogs of the non-symmetric examples in Theorem~\ref{th:kpmain} whose isotropy is of cohomogeneity two, the isotropy action is non-polar, regardless of which $G$-invariant Riemannian metric is chosen. We will use the following lemma.

\begin{lemma}\label{lm:smfct}
Let the connected simple Lie group~$G$ act isometrically on the simply connected Riemannian symmetric space~$M$. Let $M = M_0 \times M_1 \times \dots \times M_t$ be the De Rham decomposition of~$M$, where~$M_0$ is its Euclidean factor and $M_1,\dots, M_t$ are irreducible.
Assume that there is an index $i \in \{0, \dots, t\}$ such that $\dim(G) > \dim(I(M_i))$, where $I(M_i)$ is the isometry group of~$M_i$. Then the $G$-action on~$M$ acts trivially on the factor~$M_i$, i.e.\
\[
g \cdot (x_0, \dots, x_t) = (y_0, \dots, y_t) \Rightarrow x_i = y_i.
\]
\end{lemma}

\begin{proof}
Let $I_0(M)$ and $I_0(M_i)$ denote the connected components of the isometry groups of~$M$ and~$M_i$.
By \cite[Thm.~8.3.9]{wolf} we have $I_0(M) =  I_0(M_0)  \times I_0(M_1) \times \dots \times I_0(M_t)$.
It follows that every element of~$I_0(M)$ is of the form $(x_0, \dots, x_t) \mapsto (g_0 \cdot x_0, \dots, g_t \cdot x_t)$, where $g_k \in I_0(M_k)$, in particular, the $G$-action on~$M$ induces a Lie algebra homomorphism $\varphi \colon \g{g} \to \g{i}(M_i)$ from the Lie algebra of~$G$ into the Lie algebra~$\g{i}(M_i)$ of~$I(M_i)$. Since $\dim(\g{g}) > \dim(\g{i}(M_i))$, the kernel of~$\varphi$ is non-trivial and it now follows that $\varphi=0$ as $\g{g}$ is simple. Since $G$ is connected, this proves the statement of the lemma.
\end{proof}

\begin{theorem}\label{th:noncompact-nonpolar}
Let $(G,H)$ be one of $(\SU{n,1},\SU{n})$, $n \ge 2$; $(\Sp{n,1},\Sp{n}\times\U1)$, $n \ge 1$; $(\Spin{8,1},\Spin7)$ and let $G/H$ be endowed with a $G$-invariant Riemannian metric. Then this Riemannian metric is not (locally) symmetric, and the isotropy action of~$H$ on~$G/H$ is infinitesimally polar but non-polar.
\end{theorem}

\begin{proof}
First, note that the cohomogeneity of the isotropy action is two in all cases, and thus, the action of $H$ on $G/H$ is infinitesimally polar.
Assume that the isotropy action of~$H$ on~$M = G/H$ is polar, where the pair $(G,H)$ is one of the pairs in
the assertion of the theorem and where $G/H$ is endowed with some $G$-invariant Riemannian metric~$\mu$.
Since the cohomogeneity of the isotropy action is two, Lemma~\ref{lm:CohTwo} shows that $(M,\mu)$ is locally symmetric. Since in all cases $G$ is simply connected and $H$ is connected, it follows that $M$ is simply connected and hence a Riemannian globally symmetric space by~\cite[Ch.~IV, Cor.~5.7]{helgason}.
Therefore, in order to prove Theorem~\ref{th:noncompact-nonpolar} it suffices to assume that $(M,\mu)$ is symmetric and reach a contradiction.
Note that $M$ is not flat, since a semisimple Lie group cannot act transitively on a Euclidean space by isometries.

\begin{enumerate}

\item {$M=\SU{n,1}/\SU{n}$.}  The isotropy representation of this space is equivalent to the $\SU{n}$-action on~$\C^n \oplus \R$, where $\SU{n}$ acts by its standard representation on the first summand and trivially on the second. Since this action is not an s-representation, i.e.\ it is not equivalent to the action of an isotropy group $I(S)_p$ of a Riemannian symmetric space~$S$ on the tangent space~$T_pS$, cf.~\cite[Ch.~X, \S6]{helgason}, the actual action of~$I(M)_p$ on~$T_pM$ will contain the above action as a proper subaction.

    First assume the action of~$I(M)_p$ on~$T_pM$ is irreducible. Since the principal orbits of~$\SU{n}$ on~$T_pM$ are $(2n-1)$-dimensional, and $I(M)_p$ leaves the $2n$-dimensional distance spheres around the origin in~$T_pM$ invariant, it follows that~$I(M)_p$ acts transitively on the unit sphere in $T_pM$, hence $M$ is an odd-dimensional simply connected non-compact Riemannian symmetric space of rank one, therefore homothetic to~$\R \HH^{2n+1}$. Then $G/H$ is a homogeneous presentation of~$\R \HH^{2n+1}$. However, by the results of~\cite[Thm.~6]{galaev} or \cite[Thm.~3.1]{cls}, any semisimple Lie group acting transitively and effectively on~$\R \HH^{2n+1}$ is isomorphic to~\SO{2n+1,1}, leading to a contradiction for all $n \ge 2$.

    If one assumes that the action of~$I(M)_p$ on~$T_pM$ is reducible, then it follows that $M$ is isometric to a symmetric space $N \times \R$ on which $G$ acts transitively. This leads to a contradiction by Lemma~\ref{lm:smfct} since $G$ is simple and $\dim(G) > \dim(I(\R))=1$.

	\item {$M=\Sp{n,1}/\Sp{n}\times\U1$.} In this case, the isotropy representation of~$G/H$ splits as $\C^{2n} \oplus \C$, where both modules are irreducible.

    Assume $M$ is reducible; then $M=M_1 \times M_2$, where $M_2$ is a symmetric space of dimension~$2$ (reducible or irreducible). Thus the isometry group~$I(M_2)$ is at most of dimension~$3$. This leads to a contradiction by Lemma~\ref{lm:smfct} since $G$ is simple and $\dim(G) \ge 10$.

    Hence $M$ is an irreducible symmetric space whose isotropy representation has the $\Sp{n}\times\U1$-representation on~$\C^{2n} \oplus \C$ as a proper subaction. Since this representation has $\eS^{4n-1} \times \eS^1$ as principal orbits, the isotropy representation of $I(M)_p$ must act transitively on the unit sphere in~$T_pM$. Since $\dim(M)=4n+2$, we have that $M$ is an $\R \HH^{4n+2}$ or a~$\C \HH^{2n+1}$. In case $M = \R \HH^{4n+2}$ we get a contradiction with~\cite[Thm.~6]{galaev} or \cite[Thm.~3.1]{cls} as in case~(i). In case $M = \C \HH^{2n+1}$, we get a contradiction with~\cite[Thm.~4.1]{cls}, since by~\cite{cls} any connected semisimple Lie group acting transitively and effectively on~$\C \HH^{2n+1}$ is isomorphic to~$\SU{2n+1,1}$.

	\item{$M=\Spin{8,1}/\Spin7$.} The isotropy representation~$\chi(G,H)$ is the direct sum $\R^7 \oplus \R^8$, where $H=\Spin7$ acts on~$\R^7$ by its standard representation and on~$\R^8$ by its spin representation. If the symmetric space~$M$ is irreducible, then it follows as in case~(i) that $M$ is homothetic to $\R \HH^{15}$. Using the result of~\cite[Thm.~6]{galaev} or \cite[Thm.~3.1]{cls} again, we arrive at a contradiction.

    If $M$ is reducible, then it follows that $M = X \times Y$ is the Riemannian product of a $7$-dimensional symmetric space~$X$ and an $8$-dimensional symmetric space~$Y$. Since we know that $M$ is not flat, we know that $M = X \times Y$ is a decomposition of~$M$ where either both factors are irreducible, or one factor is Euclidean and the other is irreducible.
    If $X$ is flat, then $I_0(X) \cong \R^7 \rtimes \SO7$. If $X$ is non-flat then $X$ is either an $\R\HH^7$ or an~$\eS^7$, since the isotropy representation acts transitively on the unit sphere in the tangent space. In all three cases, we get a contradiction with Lemma~\ref{lm:smfct}, since $\dim(I(X))=28$, $\dim(\Spin{8,1})=36$ and $\Spin{8,1}$ is simple.
\end{enumerate}
We have shown in all cases that any $G$-invariant Riemannian metric is non-symmetric and the isotropy action is non-polar.
\end{proof}

%
%

\section{The non-semisimple case}\label{sect:Heis-DR}

In this section we prove that two important families of homogeneous spaces with non-semisimple isometry groups, namely the generalized Heisenberg groups and the non-sym\-met\-ric Damek-Ricci spaces, do not have polar isotropy actions. We briefly recall the construction of these spaces and refer to~\cite{BTV} for more details.

Let $\g{n}=\g{v}\oplus\g{z}$ be a Lie algebra, equipped with an inner product $\langle \cdot,\cdot \rangle$ such that  $\langle \g{v},\g{z}\rangle=0$, and whose Lie bracket satisfies $[\g{v},\g{v}]\subseteq\g{z}$, $[\g{v},\g{z}]=[\g{z},\g{z}]=0$. Let $m=\dim \g{z}$. Define a linear map $J\colon \g{z}\to \mathrm{End}(\g{v})$ by
\[
\langle J_Z U, V\rangle=\langle [U,V],Z\rangle, \qquad \text{for all $U$, $V\in\g{v}$, $Z\in\g{z}$.}
\]
The Lie algebra $\g{n}$ is called a \emph{generalized Heisenberg algebra} if
\[
J_Z^2=-\langle Z,Z\rangle \id_\g{v}, \qquad \text{for all $Z\in\g{z}$.}
\]
The simply connected Lie group $N$ with Lie algebra $\g{n}$, when equipped with the left-invariant metric induced by $\langle \cdot,\cdot \rangle$, is called a \emph{generalized Heisenberg group}.

The map $J$ induces a representation of the Clifford algebra $\mathrm{Cl}(\g{z},q)$ on $\g{v}$, where $q=-\langle\cdot,\cdot\rangle\rvert_{\g{z}\times\g{z}}$. Conversely, a representation of such a Clifford algebra induces a map $J$ as above, and hence, a generalized Heisenberg group. Thus, the classification of these spaces follows from the classification of Clifford modules. If $m\not\equiv 3 \mod 4$, then there is exactly one irreducible Clifford module $\g{d}$ over $\mathrm{Cl}(\g{z},q)$, up to equivalence, and each Clifford module over $\mathrm{Cl}(\g{z},q)$ is isomorphic to $\g{v}\cong\oplus_{i=1}^k \g{d}$; we denote the corresponding generalized Heisenberg group by $N(m, k)$. If $m\equiv 3 \mod 4$, then there are precisely two irreducible Clifford modules, $\g{d}_+$ and $\g{d}_-$, over $\mathrm{Cl}(\g{z},q)$, up to equivalence, and each Clifford module over $\mathrm{Cl}(\g{z},q)$ is isomorphic to $\g{v}\cong\bigl(\oplus_{i=1}^{k_+}\g{d}_+\bigr)\oplus\bigl(\oplus_{i=1}^{k_-}\g{d}_-)$; we denote the corresponding group by $N(m,k_+,k_-)$.  As vector spaces, the irreducible modules $\g{d}$ (or $\g{d}_{\pm}$) are isomorphic to the vector spaces shown in Table~\ref{table:delta}. Moreover, if $m\equiv 3 \mod 4$, then $N(m,k_+,k_-)$ is isometric to $N(m,k_+',k_-')$ if and only if $(k_+',k_-')$ equals $(k_+,k_-)$ or $(k_-,k_+)$.

\begin{table}[h]
\renewcommand{\arraystretch}{1.5}
	\begin{tabular}{ccccccccc}
		\hline
		$m$ & $8p$ & $8p+1$ & $8p+2$ & $8p+3$ & $8p+4$ & $8p+5$ & $8p+6$ & $8p+7$
		\\ \hline
		$\g{d}_{(\pm)}$ & $\R^{2^{4p}}$ & $\C^{2^{4p}}$ & $\H^{2^{4p}}$ & $\H^{2^{4p}}$
		& $\H^{2^{4p+1}}$ & $\C^{2^{4p+2}}$ & $\R^{2^{4p+3}}$ & $\R^{2^{4p+3}}$
		\\
		\hline
	\end{tabular}
	\bigskip
	\caption{Irreducible Clifford modules}\label{table:delta}
\end{table}

Let $\g{a}$ be a one-dimensional real vector space with an inner product, let $B\in\g{a}$ be a unit vector, and $\g{n}$ a generalized Heisenberg algebra. We define the orthogonal direct sum $\g{s}=\g{a}\oplus\g{n}=\g{a}\oplus\g{v}\oplus\g{z}$, which we endow with the Lie algebra structure determined by $[B, U]=\frac{1}{2}U$ and $[B,Z]=Z$, for all $U\in\g{v}$ and $Z\in\g{z}$. The simply connected Lie group $S$ with Lie algebra $\g{s}$, endowed with the left-invariant Riemannian metric given by the inner product on~$\g{s}$, is called a \emph{Damek-Ricci space}, which will be denoted by $S(m,k)$ or $S(m,k_+,k_-)$ depending on the underlying generalized Heisenberg group with Lie algebra~$\g{n}$.

Damek-Ricci spaces are homogeneous and non-symmetric except for the following cases: $S(1,k)$, which is isometric to the complex hyperbolic space $\C \mathsf{H}^{k+1}$, $S(3,k,0)\cong S(3,0,k)$, which is isometric to the quaternionic hyperbolic space $\H \mathsf{H}^{k+1}$, and $S(7,1,0)\cong S(7,0,1)$, which is isometric to the Cayley hyperbolic plane $\Ca \mathsf{H}^2$.

Before studying generalized Heisenberg groups and Damek-Ricci spaces with polar isotropy we first consider those that have polar linear isotropy representation.

\begin{theorem}\label{th:linear:H-DR}
We have:
\begin{enumerate}
	\item The space $M$ is a generalized Heisenberg group with polar linear isotropy representation if and only if $M$ is one of \label{th:linear:H}
	\begin{align*}
		&&N(1,k), &&N(2,k), &&N(3,k_+,k_-), &&N(6,1), &&N(7,1,0)\cong N(7,0,1).
	\end{align*}
	
	\item The space $M$ is a non-symmetric Damek-Ricci space with polar linear isotropy representation if and only if $M$ is one of \label{th:linear:DR}
	\begin{align*}
		&S(2,k), &&S(3,k_+,k_-)\text{ with $k_+$, $k_-\geq 1$}, &&S(6,1).
	\end{align*}
\end{enumerate}
\end{theorem}

\begin{proof}
Let $M$ denote either a generalized Heisenberg group or a non-symmetric Damek-Ricci space. In both cases, the isotropy group at the identity $e$ is given by the group of automorphisms of $M$ whose differential at $e$ is an orthonormal map~\cite[\S3.1.13 and \S4.1.12]{BTV}. Then, as follows from~\cite{Riehm} (see also~\cite{ks}), $H$, the connected component of the identity of the isotropy group of $M$, is given by
$\Spin{m}\cdot K$, where
\[
K=\begin{cases}
\SO{k}, & \text{ if } m\equiv 0,6 \mod 8,
\\
\U{k}, & \text{ if } m\equiv 1,5 \mod 8,
\\
\Sp{k}, & \text{ if } m\equiv 2,4 \mod 8,
\\
\Sp{k_+}\times \Sp{k_-}, & \text{ if } m\equiv 3 \mod 8,
\\
\SO{k_+}\times\SO{k_-}, & \text{ if } m\equiv 7 \mod 8.
\end{cases}
\]
If $m\not\equiv 3\mod 4$, $H$ acts on $\g{v}=\g{d}\otimes_\mathbb{K} \mathbb{K}^k$ via the irreducible representation $\g{d}$ of $\Spin{m}$ and the standard representation of $K$ on $\mathbb{K}^k$, where $\mathbb{K}=\mathbb{R},\mathbb{C},\mathbb{H}$ depending on whether $m\equiv 0,6$, $m\equiv 1,5$, or $m\equiv 2,4 \mod 8$, respectively. If $m\equiv 3 \mod 4$ then $H$ acts on $\g{v}=(\g{d}_+\otimes_\mathbb{K}\mathbb{K}^{k_+})\oplus(\g{d}_-\otimes_\mathbb{K}\mathbb{K}^{k_-})$ via the irreducible representations $\g{d}_{\pm}$ of $\Spin{m}$ and the standard representations of $\Sp{k_{\pm}}$ or $\SO{k_{\pm}}$, depending on whether $m\equiv 3$ or $m\equiv 7\mod 8$, respectively.

If $M$ is a generalized Heisenberg group, the isotropy representation is the action of $H=\Spin{m}\cdot K$ on $\g{v}\oplus\g{z}$ given by the direct sum of the action of $H$ on $\g{v}$ described above and the standard representation $\Spin{m}\to\SO{m}$ on $\g{z}=\R^m$, whereas if $M$ is a non-symmetric Damek-Ricci space, the isotropy representation is the action of $H=\Spin{m}\cdot K$ on $\g{a}\oplus\g{v}\oplus\g{z}$ given by the direct sum of the trivial action on $\g{a}$, the action of $H$ on $\g{v}$ described above, and the standard representation on $\g{z}=\R^m$.

\smallskip
\textit{Case \textup{(\ref{th:linear:H})}: Generalized Heisenberg groups.}
\smallskip

Let us assume that $M=N(m,k)$ or $M=N(m,k_+,k_-)$ is a generalized Heisenberg group with Lie algebra $\g{n}=\g{v}\oplus\g{z}$, whose linear isotropy representation is polar. Then, the linear isotropy representation is a reducible polar representation and, hence, must appear in Bergmann's paper~\cite{bergmann}; for convenience, we will refer instead to~\cite[Appendix 9]{FGT} and follow the terminology there.

Assume that the $H$-action on $\g{v}$ is irreducible, so that the $H$-action on $\g{n}$ has exactly two irreducible submodules. The identity component of the kernel of the $H$-action on $\g{v}$ is $H_{\g{z}}=\{e\}$, whereas for $\g{z}$ it is $H_{\g{v}}=K$. Then, the actions of $H$ and $H_{\g{z}}$ on $\g{z}$ are not orbit equivalent whenever $m\geq 2$, while the actions of $H$ and $H_{\g{v}}$ on $\g{v}$ are orbit equivalent if and only if $\g{d}=\mathbb{K}$, which happens precisely when $m\in\{1,2,3\}$. Thus, following the terminology of~\cite{FGT}, the $H$-action on $\g{n}$ is standard if and only if $m=1$, special if and only if $m\in\{2,3\}$, and exceptional (called interesting in~\cite{bergmann}) otherwise; note that the definition of exceptional representation in~\cite{FGT} does not have anything to do with the usual notions of exceptional orbits or exceptional Lie groups. For $m\in\{1,2,3\}$ it is clear from the description of the $H$-action on $\g{n}$ above that $M$ has a polar isotropy representation. Since for $m\geq 4$ the $H$-action on $\g{n}$ is exceptional, we just have to analyze if each one of the examples of exceptional reducible polar representations (listed in~\cite[Table~1]{bergmann} or~\cite[Table~9.2]{FGT}) appears as the isotropy representation of some generalized Heisenberg group $M$ with $m\geq 4$. By dimension reasons, the reducible action of $\SU{4}$ on $\R^6\oplus\R^8$, where $\SU{4}\cong\Spin{6}$ acts on $\R^6$ via the standard representation $\Spin{6}\to\SO{6}$, appears only as the isotropy representation of $N(6,1)$. The $\Spin{7}$-action on $\R^7\oplus\R^8$ appears only as the isotropy representation of $N(7,1,0)$ or $N(7,0,1)$, which are isometric. By dimension reasons, and taking into account the explicit description of $H$, one checks that the remaining exceptional actions do not appear as isotropy representations of generalized Heisenberg groups.

Now assume that the $H$-action on $\g{v}$ is reducible. Then, by~\cite[Lemma~9.5]{FGT}, we have that $\Spin{m}$ is either trivial, $\SO{2}$ or $\Sp{3}$, that is, $m\in\{1,2,3\}$. But, as we stated above, in these cases $M$ always has polar isotropy representation.

\smallskip
\textit{Case \textup{(\ref{th:linear:DR})}: Damek-Ricci spaces.}
\smallskip	

Now assume that $M=S(m,k)$ or $M=S(m,k_+,k_-)$ is a non-symmetric Damek-Ricci space. The condition of being non-symmetric means that
\[
M\notin\{S(1,k), S(3,k,0)\cong S(3,0,k), S(7,1,0)\cong S(7,0,1)\}.
\]
Since the isotropy representation of $M=S(m,k)$ (resp.\ $M=S(m,k_+,k_-)$) is the same as the one of the corresponding Heisenberg group $N(m,k)$ (resp.\ $N(m,k_+,k_-)$), plus a one-dimensional trivial submodule, it follows that the isotropy representation of $M$ is polar if and only if the isotropy representation of the corresponding Heisenberg group is polar, from which the result follows.
\end{proof}

Finally, we have

\begin{theorem}
The generalized Heisenberg groups and the non-symmetric Damek-Ricci spaces have non-polar isotropy actions.
\end{theorem}

\begin{proof}
Again, we consider these two cases separately.

\smallskip
\textit{Case \textup{(i)}: Generalized Heisenberg groups.}
\smallskip

	Let $M$ be a generalized Heisenberg group and assume, by way of contradiction, that the isotropy $H$-action on $M$ is polar with section $\Sigma$. Then $T_e\Sigma$ is a section for the isotropy representation of $M$ at $e$. Hence, $M$ is one of the Heisenberg groups given in Theorem~\ref{th:linear:H-DR}~(\ref{th:linear:H}), and the isotropy representations of these spaces have cohomogeneity two, except for $N(3,k_+,k_-)$ with $k_+$, $k_-\geq 1$, which have cohomogeneity three. Hence, $T_e\Sigma=\spann\{U,Z\}$ in the first family of spaces, or $T_e\Sigma=\spann\{U,V,Z\}$ in the second family, where $U$, $V\in\g{v}$ and $Z\in\g{z}$ are mutually orthogonal unit vectors, and $V$ is perpendicular to~$J_\g{z} U$. (This follows from the fact that the two irreducible submodules of the $H$-action on $\g{v}$ are also submodules of the $\mathrm{Cl}(\g{z},q)=\mathrm{Cl}_3$ representation on $\g{v}$, and hence, invariant under $J_Z$ for all $Z\in\g{z}$.)

	Then, $L=\Exp(\R U\oplus \R J_Z U\oplus\R Z)$, where $\Exp$ denotes the Lie exponential map, is a totally geodesic submanifold of $M$~\cite[\S3.1.9]{BTV} isometric to the Heisenberg group $N(1,1)$. Observe that $T_e\Sigma\cap T_e L=\R U\oplus\R Z$. Since the intersection of totally geodesic submanifolds is totally geodesic, we have that $\Sigma\cap L$ is totally geodesic in $M$, and hence also in $L$. Since $\Sigma\cap L=\exp_e(\R U\oplus\R Z)$, where now $\exp$ denotes the Riemannian exponential map, we have that $\Sigma\cap L$ has dimension $2$. But the Heisenberg group $L\cong N(1,1)$ does not admit totally geodesic surfaces~\cite{ST}, which yields a contradiction.
	
\smallskip
\textit{Case \textup{(ii)}: Damek-Ricci spaces.}
\smallskip	

	Now let $M$ be a Damek-Ricci space. Arguing similarly as for the Heisenberg groups, we will prove that the isotropy action of the spaces given in Theorem~\ref{th:linear:H-DR}~(\ref{th:linear:DR}) is non-polar. Assume, by contradiction, that the isotropy $H$-action on $M$ is polar with section $\Sigma$. Then $T_e\Sigma$ is a section for the isotropy representation of $M$ at $e$. Thus, $T_e\Sigma=\spann\{B,U,Z\}$ for $M\in\{S(2,k), S(5,1),S(6,1)\}$, or $T_e\Sigma=\spann\{B,U,V,Z\}$ if $M=S(3,k_+,k_-)$ with $k_+$, $k_-\geq 1$, where $B\in\g{a}$, $U$, $V\in\g{v}$, $Z\in\g{z}$ are mutually orthogonal unit vectors, and $V$ is perpendicular to~$J_\g{z} U$.
	
	Then $L=\Exp(\R B\oplus\R U\oplus \R J_Z U\oplus\R Z)$ is a totally geodesic submanifold of $M$~\cite[\S4.1.11]{BTV} isometric to a complex hyperbolic plane $\C \mathsf{H}^2$. Hence $\Sigma\cap L$ is a totally geodesic submanifold of $L$, which has dimension $3$ since $T_e\Sigma\cap T_eL=\R B\oplus\R U\oplus \R Z$. But this is impossible, since the only irreducible symmetric spaces admitting a totally geodesic hypersurface are those of constant curvature.
\end{proof}

%
%

\section{Open questions}\label{sect:open}

It is natural to ask whether Theorem~\ref{th:main} can be generalized to the non-compact case. In order to prove an analogous result in the non-compact case, one needs to decide if the spaces given in Corollaries~\ref{cor:ncslin} and~\ref{cor:nclinpoliso} can have polar isotropy actions.
For example, it remains to decide whether the space $\Sp{n,1}/\Sp{n}$ can be endowed with a Riemannian metric such that its isotropy action is polar.

More generally, one may study whether Riemannian homogeneous spaces~$G/H$ with non-trivial polar isotropy actions are symmetric without requiring~$G$ to be semisimple, i.e.\ one may seek to extend the results of~\S\ref{sect:Heis-DR} to other classes of homogeneous spaces.
We are not aware of any irreducible non-symmetric Riemannian homogeneous space with a non-trivial polar isotropy action.

Finally, it is an intriguing question whether one can find a conceptual or geometric proof linking polar isotropy with symmetry.


\bigskip

\begin{thebibliography}{99}

\bibitem{bergmann} Bergmann, I.: \emph{Reducible polar representations.} Manuscripta Math.\ {\bf 104}, 309--324 (2001)

\bibitem{BTV} Berndt, J., Tricerri, F., Vanhecke, L.: \emph{Generalized Heisenberg groups and Damek-Ricci harmonic spaces.} Lecture Notes in Mathematics \textbf{1598}, Springer (1995).

\bibitem{cls} Castrill\'{o}n L\'{o}pez, M., Gadea, P. M., Swann, A. F.: \emph{Homogeneous structures on real and complex hyperbolic spaces.} Illinois J.\ Math.\ \textbf{53}, no.~2, 561--574  (2009)

\bibitem{dadok} Dadok, J.: \emph{Polar coordinates induced by actions of compact Lie groups.} Trans.\ Amer.\ Math.\ Soc.\ {\bf 288}, 125--137 (1985)

\bibitem{DR} Damek, E., Ricci, F.: \emph{A class of nonsymmetric harmonic Riemannian spaces.} Bull.\ Amer.\ Math.\ Soc.\ (N.S.) \textbf{27}, 139--142 (1992)

\bibitem{discala} Di Scala, A. J.: \emph{Polar Actions on Berger Spheres.}  Int.\ J.\ Geom.\ Methods Mod.\ Phys.\ \textbf{3}, no.~5--6, 1019--1023  (2006)

\bibitem{deng} Deng, S.: \emph{Invariant Finsler metrics on polar homogeneous spaces}. Pacific J.\ Math.\ \textbf{247}, no.~1, 47--74 (2010)

\bibitem{eh} Eschenburg, J.-H., Heintze, E.: \emph{On the classification of polar representations.} Math.\ Z.\ {\bf 232}, no.~3, 391--398 (1999)

\bibitem{FGT} Fang, F., Grove, K., Thorbergsson, G.: \emph{Tits geometry and positive curvature.} arXiv:1205.6222v2 (Final version published in  Acta Math. \textbf{218}, no.~1, 1–-53  (2017)).

\bibitem{galaev} Galaev, A.: \emph{Isometry groups of Lobachevskian spaces, similarity transformation groups of Euclidean spaces and Lorentzian holonomy groups}. Rend.\ Circ.\ Mat.\ Palermo (2) Suppl.~{\bf 79}, 87--97 (2006)

\bibitem{hptt} Heintze, E., Palais, R., Terng, C.-L., Thorbergsson, G.: \emph{Hyperpolar actions and $k$-flat homogeneous spaces.} J.\ Reine Angew.\ Math.\ {\bf 454}, 163--179 (1994)

\bibitem{helgason} Helgason, S.: \emph{Differential geometry, Lie groups and symmetric spaces.} Academic Press (1978)

\bibitem{kollross} Kollross, A.: \emph{A classification of hyperpolar and cohomogeneity one actions.} Trans.\ Amer.\ Math.\ Soc.\ {\bf 354}, 571--612 (2002)

\bibitem{kp} Kollross, A., Podest\`a, F.: \emph{Homogeneous spaces with polar isotropy.} Manuscripta Math.\ {\bf 110}, 487--503 (2003)

\bibitem{ks} Kollross, A., Samiou, E.: \emph{Homogeneous spaces with sections.} Manuscripta Math. \textbf{116}, no.~2, 115--123 (2005)

\bibitem{ks2} Kollross, A., Samiou, E.: \emph{Manifolds with large isotropy groups.} Adv.\ in Geom.\ {\bf 11}, no.~1, 89--102 (2011)

\bibitem{pt1} Palais, R., Terng, C.-L.: \emph{A general theory of canonical forms.} Trans.\ Amer.\ Math.\ Soc.\ {\bf 300}, 771--789 (1987)

\bibitem{pt} Palais, R., Terng, C.-L.: \emph{Critical point theory and submanifold geometry.} Lecture Notes in Math.\ {\bf 1353}, Springer (1988)

\bibitem{Riehm} Riehm, C.: \emph{The automorphism group of a composition of quadratic forms.} Trans.\ Amer.\ Math.\ Soc.\ \textbf{269}, 403--414 (1982)

\bibitem{samiou} Samiou, E.: \emph{Flats and symmetry of homogeneous spaces.} Math. Z., {\bf 234}, no.~1, 145--162 (2000)

\bibitem{ST} Souam, R., Toubiana, E.: \emph{Totally umbilic surfaces in homogeneous $3$-manifolds.} Comment.\ Math.\ Helv.\ \textbf{84}, no.~3, 673--704  (2009)

\bibitem{wolf} Wolf, J.A.: \emph{Spaces of constant curvature. Sixth edition.} American Mathematical Society (2011)
\end{thebibliography}
\end{document}